\documentclass[reqno,final,11pt]{amsart}
\usepackage{times}
\usepackage[dvips]{graphicx}
\usepackage{color}
\usepackage{amssymb,amsthm,amsmath,color,fullpage,url}
\usepackage{enumerate}
\usepackage{todonotes}
\numberwithin{equation}{section}

\newtheorem{theorem}{Theorem}[section]

\newtheorem{lemma}[theorem]{Lemma}

\DeclareMathOperator{\link}{link}

\begin{document}
\title{On the Space of 2-Linkages}

\author{Guantao Chen}
\address{Department of Mathematics and Statistics, Georgia State University, Atlanta, Georgia , USA}
\author{Hein van der Holst}
  \address{Department of Mathematics and Statistics, Georgia State University, Atlanta, Georgia , USA}
\author{Serguei Norine}
 \address{Department of Mathematics and Statistics,  McGill University, Montreal, QC, H3A 2K6, Canada}
\author{Robin Thomas}
 \address{School of Mathematics, Georgia Institute of Technology, Atlanta, Georgia 30332-0160, USA}

\begin{abstract}
Let $G=(V,E)$ be a finite undirected graph. If $P$ is an oriented path from $r_1\in V$ to $r_2\in V$, we define $\partial(P) = r_2-r_1$. If $R, S\subseteq V$,  we denote by $P(G; R, S)$ the span of the set of all $\partial P\otimes \partial Q$ with $P$ and $Q$ disjoint oriented paths of $G$ connecting vertices in $R$ and $S$, respectively. By $L(R, S)$, we denote the submodule of $\mathbb{Z}\langle R\rangle\otimes\mathbb{Z}\langle S\rangle$ consisting all $\sum_{r\in R, s\in S} c(r,s)r\otimes s$ such that $c(r,r) = 0$ for all $r\in R\cap S$, $\sum_{r\in R} c(r, s) = 0$ for all $s\in S$, and $\sum_{s\in S} c(r, s) = 0$ for all $r\in R$. 

In this paper, we provide, when $G$ is sufficiently connected, characterizations when $P(G; R, S)$ is a proper subset of $L(R, S)$. 
\end{abstract}

\maketitle

%%Z(D(G,2); A)
%%S_\sigma(C, D; A)

\section{Introduction}\label{sec:intof}

Let $R$ be a finite set. By $\mathbb{Z}\langle R\rangle$, we denote the free $\mathbb{Z}$-module generated by $R$; that is, $\mathbb{Z}\langle R\rangle$ is the module of all $\sum_{r\in R} c(r) r$ with $c(r)\in\mathbb{Z}$. By $\mathbb{Z}\langle R\rangle^{\otimes 2}$, we denote the tensor product of $\mathbb{Z}\langle R\rangle$ and $\mathbb{Z}\langle R\rangle$; that is, $\mathbb{Z}\langle R\rangle^{\otimes 2}$ is the module of all $\sum_{r.s\in R,s \in R} c(r\otimes s) r\otimes s$, where $c(r\otimes s) \in \mathbb{Z}$. Observe that we can identify elements $c\in \mathbb{Z}\langle R\rangle^{\otimes 2}$ with linear forms $c : \mathbb{Z}\langle R\rangle\otimes \mathbb{Z}\langle R\rangle\to \mathbb{Z}$.

Suppose that $R$ is a finite set with a cyclic ordering $r_1,r_2,\ldots,r_n$. So the successor of $r_i$ in this cyclic ordering is $r_{i+1}$ if $i\not=n$, and $r_1$ if $i=n$. If $u,v \in R$ are distinct elements, we define 
\begin{equation*}
\link(u\otimes v; r_1,r_2,\ldots, r_n) = \begin{cases}
1 & \text{if $v$ comes after $u$ in the sequence $r_1,\ldots,r_n$}\\
-1 & \text{if $v$ comes before $u$ in the sequence $r_1,\ldots,r_n$.}
\end{cases}
\end{equation*}
Let $D(R) = \{\sum_{r,s\in R} c(r\otimes s) r\otimes s\in \mathbb{Z}\langle R\rangle^{\otimes 2}~|~c(r\otimes r) = 0\text{ for all }r\in R\}$. 
We extend $\link(\cdot; r_1,r_2,\ldots,r_n)$ linearly to $D(R)$. It is easy to see that,
for all $c\in D(R)$, 
\begin{equation*}
\link(c; r_1,r_2\ldots,r_n) = \link(c; r_n,r_1,\ldots,r_{n-1})
\end{equation*} 
if and only if, for all $r\in R$,
\begin{equation*}
 \sum_{s\in R} c(r\otimes s) = 0\text{ and }\sum_{s\in R} c(s\otimes r) = 0.
\end{equation*}
We denote by $L(R)$ the submodule of $D(R)$ of all $c\in D(R)$ such that,  for all $r\in R$,
\begin{equation*}
\sum_{s\in R} c(r\otimes s) = 0\text{ and } \sum_{s\in R} c(s\otimes r) = 0.
\end{equation*}
Then, if $c\in L(R)$, $\link(c; r_1,r_2\ldots,r_n)$ is independent of 
where we start the cyclic ordering of $r_1,r_2,\ldots,r_n$.

Let $R=\{r_1,r_2,\ldots,r_n\}$ with a cyclic ordering $r_1,r_2,\ldots,r_n$. If $u_1,u_2,v_1,v_2\in R$, then 
\begin{equation*}
c(u_1u_2; v_1v_2) := (u_1-u_2)\otimes (v_1-v_2) \in L(R).
\end{equation*}
Suppose $u_1,u_2,v_1,v_2$ are distinct.
 If $u_1,u_2,v_1,v_2$ occur in this order in $r_1,r_2,\ldots,r_n$, then 
 \begin{equation*}
 \link(c(u_1u_2; v_1v_2); r_1,r_2,\ldots,r_n) = 0,
 \end{equation*}
 while, if $u_1,v_1,u_2,v_2$  occur in this order in  $r_1,r_2,\ldots,r_n$, then 
 \begin{equation*}
 \link(c(u_1u_2; v_1v_2); r_1,r_2,\ldots,r_n) = 2,
 \end{equation*} 
 and if $u_2,v_1,u_1,v_2$ occur in this order in  $r_1,r_2,\ldots,r_n$, then 
 \begin{equation*}
 \link(c(u_1u_2;v_1v_2); r_1,r_2,\ldots,r_n) = -2.
 \end{equation*}
If $u_1,u_2,u_3\in R$, then 
\begin{equation*}
c(u_1,u_2,u_3) := u_1\otimes u_2 - u_1\otimes u_3 + u_2\otimes u_3 - u_2\otimes u_1 + u_3\otimes u_1 - u_3\otimes u_2\in L(R).
\end{equation*} 
Suppose $u_1,u_2,u_3$ are distinct.
If $u_1,u_2,u_3$ occur in this order in $r_1,r_2,\ldots,r_n$, then 
\begin{equation*}
\link(c(u_1,u_2,u_3); r_1,r_2,\ldots,r_n) = 2,
\end{equation*}
while, if $u_3,u_2,u_1$ occur in this order in $R$, then 
\begin{equation*}
\link(c(u_1,u_2,u_3); r_1,r_2,\ldots,r_n) = -2.
\end{equation*}

If $R, S$ are finite sets, we denote by $L(R, S)$ the submodule of $L(R\cup S)$ of all $c\in L(R\cup S)$ with $c(r,s)=0$ if $r\not\in R$ or $s\not\in S$. If $u_1,u_2\in R$ and $s_1,s_2\in S$ are distinct, then $c(u_1u_2;v_1v_2)\in L(R, S)$. Observe that $c(u_1,u_2,u_3)\in L(R,S)$ only if $u_1,u_2,u_3\in R\cap S$.

Let $G=(V,E)$ be a finite graph. If $P$ is an oriented path in $G$ from $u_2$ to $u_1$, then 
\begin{equation*}
\partial(P) := u_1-u_2 \in \mathbb{Z}\langle V\rangle.
\end{equation*}
If $R_1, R_2\subseteq V$, an \emph{$(R_1; R_2)$-linkage} in $G$ is a pair of disjoint paths $(P_1, P_2)$ in $G$ such that, for $i=1,2$, $P_i$ connects vertices in $R_i$. Let $u_1,u_2,v_1,v_2$ be vertices in $V$. A \emph{$(u_1u_2; v_1v_2)$-linkage} in $G$ is a pair of disjoint paths $(P_1, P_2)$ in $G$ such that $P_1$ connects $u_1$ and $u_2$, and $P_2$ connected $v_1$ and $v_2$ . If $R_1, R_2\subseteq V$, we denote by $P(G; R_1, R_2)$ the span of all $\partial(P_1)\otimes\partial(P_2)\in \mathbb{Z}\langle R_1\rangle\otimes\mathbb{Z}\langle R_2\rangle$ where $(P_1, P_2)$ are $(R_1; R_2)$-linkages in $G$.  Here we give $P_1$ and $P_2$ arbitrary orientations. Clearly, $P(G; R_1, R_2) \subseteq L(R_1, R_2)$. 

A \emph{separation} of a graph $G$ is a pair $(G_1, G_2)$ of subgraphs with $G_1\cup G_2 = G$ and $E(G_1)\cap E(G_2)=\emptyset$. The \emph{width} of a separation $(G_1, G_2)$ is $k := V(G_1\cap G_2)$. We call a separation of width $k$ a $k$-separation.

Let $R,S\subseteq V$. We call a $2$-separation $(G_1,G_2)$ of $G$ a \emph{sided $(R,S)$-separation} if  $|V(G_2) - V(G_1)| > 0$, $(V(G_1)-V(G_2))\cap S = \emptyset$, and $G_1-V(G_2)$ contains at least two vertices $u_1, u_2$ from $R$ and there are $(u_1w_1; u_2w_2)$- and $(u_1w_2; u_2w_1)$-linkages in $G_1$, where $\{w_1,w_2\}=V(G_1\cap G_2)$. A  graph $G$ is \emph{$(R,S)$-connected} if $G$ is $2$-connected and $G$ has no sided $(R,S)$- and no sided $(S,R)$-separations.

Let $G=(V,E)$ be a graph. If $X\subseteq V$, we denote by $N_G(X)$ the set of vertices in $V-X$ that are adjacent to a vertex in $X$. If it is clear what graph we mean, we also write $N(X)$ for $N_G(X)$.
We call a family $\mathcal{A} = \{A_1,\ldots,A_k\}$ of pairwise disjoint subsets of $V$ \emph{separated} if for $1\leq i\not=j\leq k, N(A_i)\cap A_j=\emptyset$. If $\mathcal{A}$ is a separated family of subsets of $V$, we denote by $p(G, \mathcal{A})$ the graph  obtained from $G$ by deleting $A_i$ for each $i$ and adding new edges joining every pair of distinct vertices in $N(A_i)$.

A \emph{nearly planar} graph is a pair $(G, \mathcal{A})$, where $G=(V,E)$ is a graph and $\mathcal{A} = \{A_1,\ldots,A_k\}$ is a separated family of subsets of $V$ such that for each $A_i$, $|N(A_i)|\leq 3$ and 
$p(G, \mathcal{A})$ can be drawn in a closed disc $D$ with no pair of edges crossing and for each $A_i$ with $|N(A_i)|=3$, $N(A_i)$ induces a facial triangle in $p(G, \mathcal{A})$. If, in addition, $r_1,\ldots,r_n$ are vertices in $G$ such that $r_i\not\in A_j$ for any $A_j\in \mathcal{A}$ and $r_1,\ldots,r_n$ occur on the boundary of $D$ in the specified order, then we say that $(G, \mathcal{A},r_1,\ldots,r_n)$ is \emph{nearly planar}.
If there is no need to specify the family $\mathcal{A}$, we also write $(G, r_1, \ldots,r_n)$ for $(G, \mathcal{A},r_1,\ldots,r_n)$. 

Let $t_1,\ldots,t_n$ be a cyclic ordering and suppose $R, S\subseteq \{t_1,\ldots,t_n\}$ such that $R\cup S=\{t_1,\ldots,t_n\}$. We say that $R$ and $S$ \emph{cross} if $R=S$ and $|R|=3$, or there exist vertices $u_1,u_2\in R$ and vertices $v_1,v_2\in S$ such that $u_1,v_1,u_2,v_2$ are distinct and occurring in this order in the cyclic ordering $t_1,\ldots,t_n$.

In this paper, we prove the following theorem.

\begin{theorem}\label{thm:K33}
Let $G=(V,E)$ be a graph, let $R,S\subseteq V$ and suppose $G$ is $(R,S)$-connected. Then the following are equivalent:
\begin{enumerate}
\item\label{item:i1} $L(R, S) - P(G; R, S)\not=\emptyset$;
\item\label{item:i2} there exists a cyclic ordering $t_1,\ldots,t_n$ of the vertices in $R\cup S$ such that $R$ and $S$ cross and $(G, t_1,\ldots,t_n)$ is nearly planar; and
\item\label{item:i3} there exists a cyclic ordering $t_1,\ldots,t_n$ of the vertices in $R\cup S$ such that $R$ and $S$ cross and, for all $z\in L(R, S)$, $\link(z; t_1,\ldots, t_n) = 0$ if and only if $z \in P(G; R, S)$.
\end{enumerate}
\end{theorem}

\section{Elementary tensors}
An \emph{elementary $(R,S)$-tensor} is a tensor of the form $c(u_1u_2;v_1v_2)$ with $u_1,u_2\in R$ and $v_1,v_2\in S$ and $u_1,u_2,v_1,v_2$ distinct, or a tensor of the form $c(u_1,u_2,u_3)$ with $u_1,u_2,u_3\in R\cap S$ and $u_1,u_2,u_3$ distinct.
In this section, we show that each element in $L(R, S)$ can be written as a sum of elementary $(R,S)$-tensors. The proof follows the proof of Theorem~7 in \cite{MR2482969}.

\begin{lemma}\label{lem:circuitandkurcycles}
Let $R,S$ be finite sets. Then $L(R, S)$ is generated by elementary $(R, S)$-tensors.
\end{lemma}
\begin{proof}
Order the elements in $R$ as $r_1,\ldots, r_k$ such that we start with the elements that also belong to $S$. Similary, we order the elements in $S$ as $s_1,\ldots,s_m$ such that we start with the elements that also belong to $R$.
For every $c = \sum_{r\in R, s\in S} c(r\otimes s) r\otimes s \in L(R, S)$, we define the height of $c$ as the minimum of $i+j$ over $c(r_i\otimes s_j)\not=0$.  Suppose for a contradiction that there exists a $c\in L(R, S)$ that is not a sum of elementary $(R, S)$-tensors. Choose such a $c$ whose height is maximal. Then the height of $c$ is not greater than $k+m$.

Let $i$ and $j$ with $c(r_i\otimes s_j)\not=0$ and $i+j$ minimum. Let $r_{i'}\in R$ be an element unequal to $r_i$ such that $c(r_{i'}\otimes s_j)\not=0$, and let $s_{j'} \in S$ be an element unequal to $s_j$ such that $c(r_i\otimes s_j')\not=0$. Then $r_i\not=s_j$, $r_i\not=s_{j'}$, and $r_{i'}\not=s_j$.

Suppose that $r_{i'}\not=s_{j'}$. Then $c' = c-c(r_i\otimes s_j)c(r_ir_{i'};s_js_{j'})\in L(R, S)$ is not a sum of elementary $(R, S)$-tensors. As the height of $c'$ is larger than the height of $c$, we obtain a contradiction. 

Therefore $r_{i'}=s_{j'}$. Then, by choice of the ordering the elements $r_1,\ldots,r_k$ and $s_1,\ldots,s_m$ and by the minimality of $i+j$, $r_i\in S$ and $s_j\in R$. So each of $r_i, s_j, r_{i'}=s_{j'}$ belongs to both $R$ and $S$. Then $c' = c-c(r_i\otimes s_j)c(r_i,s_j,r_{i'})\in L(R, S)$ is not a sum of elementary $(R, S)$-tensors. As the height of $c'$ is larger than the height of $c$, we obtain a contradiction. 
\end{proof}

\section{Sums of Linkages}

Let $G=(V,E)$ be a graph and let $u_1,u_2, v_1,v_2$ be distinct vertices of $G$. In this section, we show that the existence of certain linkages in $G$ implies that $c(u_1u_2; v_1v_2)\in P(G; R,S)$.

\begin{lemma}\label{lem:K23d0linkage}
	Let $G=(V,E)$ be a graph and let $R, S\subseteq V$.  Let  $u_1,u_2, u\in R$ and $v_1,v_2\in S$. If $G$ contains a $(u_1u; v_1v_2)$- and a $(uu_2; v_2v_1)$-linkage, then $c(u_1u_2; v_1v_2)\in P(G; R, S)$.
\end{lemma}
\begin{proof}
	Since $G$ has a $(u_1u; v_1v_2)$-linkage, $c(uu_1;v_2v_1)\in P(G; R, S)$.
	Since $G$ has a $(uu_2; v_2v_1)$-linkage, $c(uu_2;v_1v_2)\in P(G; R, S)$.
	Then $c(u_1u_2;v_1v_2) =  c(u_1u;v_1v_2)+c(uu_2;v_1v_2)\in P(G; R, S)$.
\end{proof}

\begin{lemma}\label{lem:misses1ofv1v2linkage}
	Let $G=(V,E)$ be a graph and let $R, S\subseteq V$. Let $u_1,u_2, u\in R$ and $v_1,v_2, v\in S$. If $G$ contains a $(u_1u;v_1v_2)$-, a $(u u_2;vv_1)$-, and a $(u_2u; v_2v)$-linkage, then $c(u_1u_2; v_1v_2)\in P(G; R,S)$.
\end{lemma} 
\begin{proof}
	Since $G$ contains a $(u_1u;v_1v_2)$-linkage, $c(uu_1; v_2v_1)\in P(G; R, S)$.
	Since $G$ contains a $(u u_2;vv_1)$-linkage, $c(u_2u;v_1v)\in P(G; R, S)$.
	Since $G$ contains a $(u_2u; v_2v)$-linkage, $c(uu_2;vv_2)\in P(G; R, S)$.
	Then $c(u_1u_2; v_1v_2) = c(uu_1; v_2v_1) + c(u_2u;vv_1) +c(uu_2;vv_2) \in P(G; R,S)$. 
\end{proof}

\begin{lemma}\label{lem:2sidea}
	Let $G=(V,E)$ be a graph and let $R,S\subseteq V$. Let $u_1,u_2,u\in R$ and $v_1,v_2,v',v''\in S$. If $G$ contains a $(u_2u_1;v'v_1)$-, a $(u u_2; v' v_2)$-, a $(u u_1; v' v_1)$-, a $(u u_1; v'' v_2)$-, and a $(u u_1; v'' v_1)$-linkage, then $c(u_1u_2;v_1v_2)\in P(G; R,S)$.
\end{lemma}
\begin{proof}
	Since $G$ contains a $(u_1u_2;v'v_1)$-linkage, $c(u_1u_2;v_1v')\in P(G; R, S)$.
	Since $G$ contains a $(u u_2; v' v_2)$-linkage, $c(u_2u;v_2v')\in P(G; R, S)$.
	Since $G$ contains a $(u u_1; v' v_1)$-linkage, $c(uu_1;v_1v')\in P(G; R, S)$.
	Since $G$ contains a $(u u_1; v'' v_2)$-linkage, $c(uu_1; v_2v'') \in P(G; R, S)$.
	Since $G$ contains a $(u u_1; v'' v_1)$-linkage, $c(u_1u;v_1v'')\in P(G; R, S)$.
	Then $c(u_1u_2;v_1v_2) = c(u_1u_2;v_1v')+ c(u_2u;v_2v') + c(uu_1;v_1v') + c(uu_1; v_2v'') +c(u_1u;v_1v'') \in P(G; R, S)$.
\end{proof}

\begin{lemma}\label{lem:2sideb}
	Let $G=(V,E)$ be a graph and let $R,S\subseteq V$. Let $u_1,u_2\in R\cap S$, let $v_1,v_2\in S$, and let  $u\in R$. If $G$ contains a $(u_1u; v_1u_2)$-, a $(u u_2; u_1v_1)$-, a $(u u_1; u_2 v_2)$-, and a $(u u_2; u_1 v_2)$-linkage, then $c(u_1u_2;v_1v_2)\in P(G; R,S)$.
\end{lemma}
\begin{proof}
	Since $G$ contains a $(u_1u; v_1u_2)$-linkage, $c(uu_1;vu_2)\in P(G; R, S)$.
	Since $G$ contains a $(u u_2; u_1v_1)$-linkage, $c(uu_2;v_1u_1)\in P(G; R, S)$.
	Since $G$ contains a $(u u_1; u_2 v_2)$-linkage, $c(uu_1;v_2u_2)\in P(G; R, S)$.
	Since $G$ contains a $(u u_2; u_1 v_2)$-linkage, $c(uu_2;v_2u_1)\in P(G; R, S)$.
	Then $c(u_1u_2;v_1v_2) = c(u_1u;vu_2)+ c(uu_2;v_1u_1) -c(uu_1;v_2u_2) + c(uu_2;v_2u_1)\in P(G; R, S)$.
\end{proof}

\section{Linkages in Nearly Planar Graphs}

In this section, we prove the following theorem.

\begin{theorem}\label{thm:3planarboundary}
Let $G$ be a $2$-connected graph and let $R, S\subseteq V(G)$. If $L(R,S)-P(G; R, S)\not=\emptyset$, then either $G$ has a sided $(R,S)$-separation or a side $(S,R)$-separation, or there exists a cyclic ordering $w_1,\ldots,w_{n}$ of the vertices in $R\cup S$ such that $(G, w_1,\ldots,w_{n})$ is nearly planar and $R$ and $S$ cross.
\end{theorem}

If $P$ is a path in a graph $G=(V,E)$ connecting $u$ and $v$, we also write that $P$ is a $u\to v$ path in $G$.
If $S\subseteq V$ and $P$ is a path in $G$ connecting $u$ and a vertex in $S$, we also write that $P$ is a $u\to S$ path in $G$.

Let $(G, \mathcal{A})$ be a nearly planar graph. We call an $A\in \mathcal{A}$ \emph{minimal} if there is no collection $\mathcal{H}$ of pairwise disjoint subsets of $A$ such that $\mathcal{H}\not=\{A\}$ and $(G[A\cup N(A)], \mathcal{H}, a_1, \ldots, a_m)$ is nearly planar, where $\{a_1,\ldots,a_m\}=N(A)$.
If every element of $\mathcal{A}$ is minimal, we say that $\mathcal{A}$ is \emph{minimal}.

\begin{lemma}\label{lem:Aminimal}
Let $(G, \mathcal{A}; r_1,\ldots,r_n)$ be a nearly planar graph. Let $A\in \mathcal{A}$ with $|N(A)|=3$ and $A$ minimal, let $\{u_1,u_2,u_3\} = N(A)$, and let $u\in A$. Then $G[A\cup N(A)]$ contains a $(u_1u_2; u_3u)$-, a $(u_1u_3; u_2u)$-, and a $(u_1u; u_2u_3)$-linkage.
\end{lemma}

\begin{lemma}\label{lem:linkageliftH}
Let $(G, \mathcal{A})$ be a nearly planar graph, where $\mathcal{A}$ is minimal. Let $H = p(G, \mathcal{A})$ and let $u_1, u_2, v_1, v_2$ be vertices of $H$. If $H$ contains a $(u_1u_2; v_1v_2)$-linkage, then $G$ contains a $(u_1u_2; v_1v_2)$-linkage.
\end{lemma}
\begin{proof}
Let $(P, Q)$ be a $(u_1u_2; v_1v_2)$-linkage in $H$. Replace each edge $e$ of $P$ and $Q$ whose ends $u_e$ and $v_e$ belong to $N(A)$ for some $A\in \mathcal{A}$ by a $u_e\to v_e$ path $T$ in $G[A\cup N(A)]$. Let $P'$ and $Q'$ be the resulting paths. Then $(P', Q')$ is a $(u_1u_2; v_1v_2)$-linkage in $G$.
\end{proof}

\begin{lemma}\label{lem:linkageliftG}
Let $(G, \mathcal{A})$ be a nearly planar graph, where $\mathcal{A}$ is minimal. Let $H = p(G, \mathcal{A})$.
Let $u_1,u_2,v_1,v_2$ be vertices of $G$, with $u_1,u_2,v_1\in V(H)$ and $v_2\not\in V(H)$. Let $P$ be a $u_1\to u_2$ path in $H$, and let $A\in \mathcal{A}$ with $v_2\in A$. Let $r\in N(A)-V(P)$ and suppose $Q$ is a $v_1\to r$ path disjoint from $P$. Then $G$ has a $(u_1u_2; v_1v_2)$-linkage.
\end{lemma}
\begin{proof}
Suppose $e$ is an edge of $P$ whose ends $u_e$ and $v_e$ belongs to $N(A)$.  Since $\mathcal{A}$ is minimal, $A$ is minimal, and so $G[A\cup N(A)]$ contains a $(u_ev_e; ur)$-linkage $(T_1,T_2)$. Replace the edge $e$ in $P$ by $T_1$. Replace any other edge $f$ of $P$ whose ends $s_f$ and $t_f$ belong to $N(A')$ for some $A'\in \mathcal{A}$ with $A'\not=A$ by a $s_f\to t_f$ path $T$ in $G[A'\cup N(A')]$. Let $P'$ be the resulting path. Replace each edge $f$ of $Q$  whose ends $s_f$ and $t_f$ belong to $N(A)$ for some $A\in \mathcal{A}$ by a $s_f\to t_f$ path $T$ in $G[A\cup N(A)]$. Let $Q'$ be the resulting path. Then $(P', Q'\cup T_2)$ is a $(u_1u_2; v_1v_2)$-linkage in $G$.
\end{proof}

Let $C$ be an oriented cycle of a graph $G$. If $u,v\in V(C)$, we denote by $C[u,v]$ the path in $C$ when traversing $C$ from $u$ to $v$ in forward direction. By $C(u,v)$ we denote $C[u,v]-\{u,v\}$. 

\begin{lemma}\label{lem:K23d0}
Let $(G, \mathcal{A})$ be a nearly planar graph, where $G$ is $2$-connected and $\mathcal{A}$ is minimal, and let $R, S\subseteq V(G)$. Let $H = p(G, \mathcal{A})$ and let $C$ be an oriented cycle of $H$. Suppose that $u_1,u_2,v_1, v_2$ are distinct vertices with $u_1,u_2\in R$  and $v_1,v_2\in S$ such  that $u_1,v_1,u_2,v_2$ occurs in this order on $C$.  Let $B$ be a component of $G-V(C)$ that has a neighbor in $C(v_1,v_2)$ and a neighbor in $C(v_2,v_1)$. If $B$ contains a vertex of $R$, then $c(u_1u_2;v_1v_2)\in P(G; R, S)$.
\end{lemma}
\begin{proof}
Let $u\in R$ belong to $B$. 
Suppose first that $u\not\in A$ for all $A\in \mathcal{A}$. Let $P'_1$ be a $u_1\to u$ path in $H-V(C[v_1,v_2])$, and let $P'_2$ be a $u\to u_2$ path in $H - V(C[v_2,v_1])$. Then $(P'_1, C[v_1,v_2])$ and $(P'_2, C[v_2,v_1])$ are $(u_1u, v_1v_2)$- and $(uu_2;v_2v_1)$-linkages in $H$, respectively. By Lemma~\ref{lem:linkageliftH}, $G$ contains a $(u_1u, v_1v_2)$- and a $(uu_2;v_2v_1)$-linkage. By Lemma~\ref{lem:K23d0linkage}, $c(u_1u_2;v_1v_2)\in P(G; R, S)$.

Suppose next that $u\in A$ for some $A\in \mathcal{A}$.  Since $B$ has a neighbor in $C[v_2,v_1]$ and a neighbor in $C[v_1,v_2]$, $N(A)-V(C[v_2,v_1])\not=\emptyset$ and $N(A)-V(C[v_1,v_2])\not=\emptyset$. Let $r\in N(A)-V(C[v_2,v_1])$. Since $B$ has a neighbor in $C[v_2,v_1]$ and a neighbor in $C[v_1,v_2]$, there exists an $r\to u_2$ disjoint from $C[v_2,v_1]$. Hence $H$ contains a $(v_1v_2,ru_2)$-linkage. By Lemma~\ref{lem:linkageliftG}, $G$ contains a $(v_1v_2, uu_2)$-linkage in $G$. In the same way, $G$ contains a $(v_1v_2,uu_1)$-linkage. By Lemma~\ref{lem:K23d0linkage}, $c(u_1u_2;v_1v_2)\in P(G; R, S)$. 
\end{proof}

\begin{lemma}\label{lem:misses1ofv1v2}
Let $(G, \mathcal{A}, t_1,\ldots,t_n)$ be a nearly planar graph, where $G$ is $2$-connected and $\mathcal{A}$ is minimal, and let $R, S\subseteq V(G)$ such that $\{t_1,\ldots,t_n\}\subseteq R\cup S$. Suppose $u_1,u_2,v_1,v_2$ are distinct vertices with $u_1,u_2\in R$ and $v_1,v_2\in S$ such that $u_1,v_1,u_2,v_2$ occur in this order in the cyclic ordering $t_1,\ldots,t_n$. Let $(H, t_1,\ldots,t_n) = p(G, \mathcal{A}, t_1,\ldots,t_n)$ and let $C$ be the cycle bounding the infinite face of a plane drawing of $H$ with the vertices $t_1,\ldots,t_n$ on $C$. Let $B$ be a component of $G-V(C)$ such that $N(B)\subseteq C[v_2,v_1]$, with at least one neighbor of $B$ in $C(v_2,v_1)$. Then at least one of the following holds:
\begin{enumerate}
\item $c(u_1u_2;v_1v_2)\in P(G; R, S)$, 
\item $G$ has a sided $(R,S)$-separation or a sided $(S,R)$-separation,
\item  there exists a cyclic ordering $w_1,\ldots,w_{n+1}$ with $\{t_1,\ldots,t_n\}\subset \{w_1,\ldots,w_{n+1}\}\subseteq R\cup S$ such that $t_1,\ldots,t_n$ occur in this ordering in $w_1,\ldots,w_{n+1}$ and $(G, w_1,\ldots,w_{n+1})$ is nearly planar, or
\item no vertex of $B$ belongs to $R\cup S$.
\end{enumerate}
\end{lemma} 
\begin{proof}
For each component $K$ of $G-V(C)$ such that $N(K)\subseteq C[v_2,v_1]$ and at least one neighbor belongs to $C(v_2,v_1)$, let $r_i(K)$, $i=1,2$, be the neighbor of $K$ in $C[v_2,v_1]$ that is closer to $v_i$. Choose a component $B$ of $G-V(C)$ such that the length $C[r_2(B), r_1(B)]$ is minimum. For $i=1,2$, let $r_i = r_i(B)$. We may assume that a vertex $u\in R\cup S$ belongs to $B$. Let $\mathcal{A}'$ be the the family of all sets $A\in\mathcal{A}$ such that $A\subseteq V(B)$ and let $B' = p(B, \mathcal{A}')$.

Suppose first that $C(r_2,r_1)$ does not contain a vertex of $R\cup S$. For $i=1,2$, let $z_i$ be the vertex of $\{t_1,\ldots,t_n\}$ in $C[r_1,r_2]$ closer to $r_i$. Then $C(z_1,z_2)$ has no vertices of $\{t_1,\ldots,t_n\}$. Let the cyclic ordering $w_1,\ldots,w_{n+1}$ be obtained from $t_1,\ldots,t_n$ by adding $u$ between $z_1$ and $z_2$. Then $(G, w_1,\ldots,w_{n+1})$ is nearly planar, $\{t_1,\ldots,t_n\}\subset\{w_1,\ldots,w_{n+1}\}\subseteq R\cup S$, and $t_1,\ldots,t_n$ occur in this ordering in $w_1,\ldots,w_{n+1}$ .

We may therefore assume that $C(r_2,r_1)$ contains at least one vertex of $R\cup S$.

Suppose that $u\in R$ and $C(r_2,r_1)\cap S\not=\emptyset$, or that $u\in S$ and $C(r_2,r_1)\cap R\not=\emptyset$. By symmetry, we may assume that the former holds; let $v$ be a vertex in $C(r_2,r_1)\cap S$.

Suppose first that $u\not\in A$ for every $A\in \mathcal{A}$.  Let $P_0$ be a $u\to u_1$ path in $H-V(C[v_1,v_2])$ and let $Q_0 = C[v_1,v_2]$. Then $(P_0, Q_0)$ is a $(uu_1; v_1v_2)$-linkage in $H$.  By Lemma~\ref{lem:linkageliftH}, $G$ contains a $(uu_1; v_1v_2)$-linkage. Let $P$ be an $r_2\to u$ path in $B'$ that is internally disjoint from $C[r_2,r_1]$ and let $P_1$ be the concatenation of $P$ and $C[u_2,r_2]$. Let $Q_1=C[v_1,v]$. Then $(P_1,Q_1)$ is a $(u u_2;v_1 v)$-linkage in $H$. By Lemma~\ref{lem:linkageliftH}, $G$ contains a $(uu_2; v_1v)$-linkage. In the same way, we define the $(u u_2; v_2 v)$-linkage $(P_2, Q_2)$. By Lemma~\ref{lem:misses1ofv1v2linkage}, $c(u_1u_2;v_1v_2)\in P(G; R, S)$.

We may therefore assume that $u\in A$ for some $A\in \mathcal{A}$.
Suppose that $N(A)-V(C)=\emptyset$. Then $N(A)\subseteq V(C[v_2,v_1])$ and $r_1,r_2\in N(A)$. Since $v\in C(r_1,r_1)$, $|N(A)|=3$. Let $r \in N(A)-\{r_1,r_2\}$ and let $P$ be a $r\to u_1$ path in $C(v_2,v_1)$. Then $(P, C[v_1,v_2])$ is a $(ru_1; v_1v_2)$-linkage in $H$. By Lemma~\ref{lem:linkageliftG}, there exists a $(uu_1; v_1v_2)$-linkage in $G$. Since $(C[r_1,u_2], C[v_2,v])$ is a $(r_1u_2; vv_2)$-linkage in $H$, $G$ contains a $(uu_2; v_2v)$-linkage, by Lemma~\ref{lem:linkageliftG}. Since $(C[u_2,r_2], C[v,v_1])$ is a $(u_2r_2; vv_1)$-linkage in $H$, $G$ contains a $(u_2u; vv_1)$-linkage. By Lemma~\ref{lem:misses1ofv1v2linkage}, $c(u_1u_2;v_1v_2) \in P(G; R, S)$.

Suppose next that $N(A)-V(C)\not=\emptyset$. Let $r\in N(A)-V(C)$.  Let $P$ be a $r\to r_1$ path in $B'$ and let $P_0$ be the concatenation of $P$ and $C[r_1,u_1]$. Then $(P_0, C[v_1,v_2])$ is a $(ru_1; v_1v_2)$-linkage in $H$. By Lemma~\ref{lem:linkageliftG}, $G$ contains a $(uu_1; v_1v_2)$-linkage. Let $P'$ be a $r\to r_1$ path in $B'$ and let $P_1$ be the concatenation of $P'$ and $C[r_1,u_2]$. Then $(P_1, C[v_2,v])$ is a $(ru_2; vv_2)$-linkage in $H$. By Lemma~\ref{lem:linkageliftG}, $G$ contains a $(uu_2; vv_2)$-linkage. Let $P''$ be a $r\to r_2$ path and let $P_2$ be the concatenation of $P''$ and $C[u_2,r_2]$. Then $(P_2,C[v,v_1])$ is a $(ru_2; vv_1)$-linkage in $H$. By Lemma~\ref{lem:linkageliftG}, $G$ contains a $(uu_2; vv_1)$-linkage. By Lemma~\ref{lem:misses1ofv1v2linkage}, $c(u_1u_2;v_1v_2) \in P(G; R, S)$.

Suppose next that $B$ and $C(r_2,r_1)$ contain no vertex of $R$, but that $C(r_2,r_1)$ contains a vertex of $S$, or that $B$ and $C(r_2,r_1)$ contain no vertex of $S$, but that $C(r_2,r_1)$ contains a vertex of $R$.
 We may assume that $B$ and $C(r_2,r_1)$ contain no vertex of $R$, but that $C(r_2,r_1)$ contains a vertex $v\in S$. Since $B$ contains a vertex $R\cup S$, $B$ contains a vertex $u\in S$. Let $V_1$ be the set of vertices $w$ in $G$ such that there exists a $w\to \{u,v\}$ path that is internally disjoint from $\{r_1,r_2\}$. Let $V_2 = V(G)-(V_1-\{r_1,r_2\})$. Let $G_1=G[V_1]$ and let $G_2=G[V_2]$. Then $(G_1,G_2)$ is a $2$-separation. It remains to show that $G_1$ contains a $(r_1u, r_2v)$- and a $(r_2u; r_1v)$-linkage.

Suppose $u\not\in A$ for all $A\in \mathcal{A}$. Since $C[r_2,r_1]\cup B'$ contains a $(ur_1; vr_2)$-linkage, $G_1$ contains a $(ur_1; vr_2)$-linkage. In the same way, $G_1$ contains a $(ur_2; vr_1)$-linkage. Suppose next that $u\in A$ for some $A\in \mathcal{A}$. If $N(A)-V(C)\not=\emptyset$, let $r\in N(A)-V(C)$. Since $C[r_2,r_1]\cup B'$ contains a $(rr_1; vr_2)$-linkage, $G_1$ contains a $(ur_1; vr_2)$-linkage, by Lemma~\ref{lem:linkageliftG}. In the same way, $G_1$ contains a $(ur_2; vr_1)$-linkage. Suppose next that $N(A)-V(C)=\emptyset$. Let $r_1,r_2 \in N(A)$. Let $r_3\in N(A)-\{r_1,r_2\}$. By Lemma~\ref{lem:Aminimal}, $G[A\cup N(A)]$ contains a $(ur_2; r_3r_1)$-linkage and a $(ur_1; r_3r_2)$-linkage. Hence $G_1$ contains a $(ur_2; vr_1)$- and a $(ur_1; vr_2)$-linkage.
\end{proof}

\begin{lemma}\label{lem:2side}
Let $(G, \mathcal{A})$ be a nearly planar graph, where $G$ is $2$-connected, and let $R, S\subseteq V(G)$. Let $H = p(G, \mathcal{A})$ and let $C$ be an oriented cycle of $H$. Suppose $u_1,v_1,u_2,v_2$ are distinct vertices with $u_1,u_2\in R$  and $v_1,v_2\in S$ such that $u_1,v_1,u_2,v_2$ occurs in this order on $C$. Let $B$ be a component of $G-V(C)$ such that $N(B)=\{v_1,v_2\}$.
If $B$ contains a vertex of $R$ and both $C(v_2,v_1)$ and $C(v_1,v_2)$ contain a vertex of $S$, then 
$c(u_1u_2;v_1v_2)\in P(G; R,S)$. 
\end{lemma}
\begin{proof}
Let $u\in V(B)\cap R$, and let $v', v''\in S$ be vertices in $C(v_2,v_1)$ and $C(v_1,v_2)$, respectively. 

We first assume that $u_1\not\in S$ or $u_2\not\in S$. By symmetry, we may assume that $u_1\not\in S$. By symmetry, we may assume that $v'$ is on $C(u_1,v_1)$.
Let $P_1=C[u_2,u_1]$ and let $Q_1=C[v',v_1]$. Hence $G$ contains a $(u_2u_1; v'v_1)$-linkage.
Let $S_2$ be a $u\to v_1$ path in $B\setminus \{v_2\}$, and let $P_2$ be the concatenation of $C[v_1,u_2]$ and $S_2$. Let $Q_2=C[v_2,v']$. Hence $G$ contains a $(u u_2; v'v_2)$-linkage. Let $S_3$ be a $u\to v_2$ path in $B\setminus \{v_1\}$, and let $P_3$ be the concatenation of $C[v_2,u_1]$ and $S_3$. Let $Q_3=C[v',v_1]$. Hence $G$ contains a $(u u_1; v'v_1)$-linkage.
Let $S_4$ be a $u\to v_1$ path in $B\setminus \{v_2\}$, and let $P_4$ be the concatenation of $S_4$ and $C[u_1,v_1]$.  Let $Q_4=C[v'',v_2]$. Hence $G$ contains a $(u u_1; v''v_2)$-linkage.
Let $S_5$ be a $u\to v_2$ path in $B\setminus\{v_1\}$, and let $P_5$ be the concatenation of $S_5$ and $C[u_1,v_2]$.  Let $Q_5=C[v_1,v'']$. Hence $G$ contains a $(u u_1; v''v_1)$-linkage.
By Lemma~\ref{lem:2sidea}, $c(u_1u_2;v_1v_2)\in P(G; R, S)$.
%%check above again

%%%below done. Check again.
Next we assume that $u_1, u_2\in S$. Let $P_1$ be a $u_1\to u$ path in $C[u_1,v_2]\cup (B\setminus \{v_1\})$ and let $Q_1=C[v_1,u_2]$. Hence $G$ contains a $(u_1u; v_1u_2)$-linkage. Let $P_2$ be a $u_2\to u$ path in $C[u_2,v_2]\cup (B\setminus \{v_1\})$ and let $Q_2 = C[u_1,v_1]$. Hence $G$ contains a $(u u_2; u_1v_1)$-linkage.
Let $P_3$ be a $u_1\to u$ path in $C[u_1,v_1]\cup (B\setminus \{v_2\})$
and let $Q_3 = C[u_2,v_2]$. Hence $G$ contains a $(u u_1; u_2 v_2)$-linkage.
Let $P_4$ be a $u_2\to u$ path in $C[v_1,u_2]\cup (B\setminus \{v_2\})$ and let $Q_4 = C[v_2,u_1]$. Hence $G$ contains a $(u u_2; u_1 v_2)$-linkage. By Lemma~\ref{lem:2sideb}, $c(u_1u_2;v_1v_2)\in P(G; R, S)$.
\end{proof}

\begin{lemma}\label{lem:midsep}
Let $(G, \mathcal{A})$ be a nearly planar graph, where $G$ is $2$-connected, and let $R, S\subseteq V(G)$. Let $H = p(G, \mathcal{A})$ and let $C$ be an oriented cycle bounding the infinite face of a plane drawing of $H$. Suppose $u_1,v_1,u_2,v_2$ are distinct vertices with $u_1,u_2\in R$  and $v_1,v_2\in S$ such that $u_1,v_1,u_2,v_2$ occurs in this order on $C$.  Let $B$ be a component of $G-V(C)$ such that $N(B)=\{v_1,v_2\}$. Then at least one of the following holds:
\begin{enumerate}[(i)]
\item $c(u_1u_2;v_1v_2)\in P(G; R,S)$,
\item $G$ has a sided $(R,S)$-separation or a sided $(S,R)$-separation, or
\item no vertex of $B$ belongs to $R\cup S$.
\end{enumerate}
\end{lemma}
\begin{proof}
Suppose first that $B$ contains a vertex of $S$. Then, by Lemma~\ref{lem:K23d0},  $c(u_1u_2;v_1v_2)\in P(G; R, S)$.

Suppose next that $B$ contains a vertex of $R$.  If both $C(v_1,v_2)$ and $C(v_2,v_1)$ contain a vertex of $S$, then, by Lemma~\ref{lem:2side}, $c(u_1u_2;v_1v_2)\in P(G; R, S)$.
Suppose $C(v_1,v_2)$ or $C(v_2,v_1)$ contains no vertex of $S$; say $C(v_2,v_1)$ contains no vertex of $S$. Let $G_1$ be the subgraph of $G$ consisting of all vertices and edges reachable by a path from $\{u_1\}\cup V(B)$ that has no vertex in $\{v_1,v_2\}$, and let $G_2 := G[(V(G)\setminus V(G_1))\cup \{v_1,v_2\}]$. Since $G_1$ contains a $(uv_1; u_1v_2)$- and a $(uv_2; u_1v_1)$-linkage, $(G_1, G_2)$ is a sided $(R,S)$-separation of $G$. The case where $C(v_1,v_2)$ contains no vertex of $S$ is similar.
\end{proof}

\begin{lemma}\label{internRandS}
Let $(G, \mathcal{A}, t_1,\ldots,t_n)$ be a nearly planar graph, where $G$ is $2$-connected and $\mathcal{A}$ is minimal, and let $R, S\subseteq V(G)$ such that $\{t_1,\ldots,t_n\}\subseteq R\cup S$. Suppose that $t_1,\ldots,t_n$ contains  distinct vertices $u_1,v_1,u_2,v_2$ in this order, where $u_1,u_2\in R$  and $v_1,v_2\in S$. Let $(H, t_1,\ldots,t_n) = p(G, \mathcal{A}, t_1,\ldots,t_n)$ and let $C$ be the cycle bounding the infinite face of a plane drawing of $H$ with $t_1,\ldots,t_n$ on $C$. If $\{t_1,\ldots,t_n\}\subset R\cup S$, then at least one of the following holds:
\begin{enumerate}[(i)]
\item $c(u_1u_2;v_1v_2)\in P(G; R, S)$,
\item there exists a sided $(R,S)$-separation or a side $(S,R)$-separation, or
\item there exists a cyclic ordering $w_1,\ldots,w_{n+1}$ with $\{t_1,\ldots,t_n\}\subset\{w_1,\ldots,w_{n+1}\}\subseteq R\cup S$ such that $t_1,\ldots,t_n$ occur in this ordering in $w_1,\ldots,w_{n+1}$ and $(G, w_1,\ldots,w_{n+1})$ is nearly planar.
\end{enumerate}
\end{lemma}

\begin{proof} 
Suppose $\{t_1,\ldots,t_n\}\subset R\cup S$. If no component $B$ of $G-V(C)$ contains a vertex of $R\cup S$, then $R\cup S\subseteq V(C)$.
We may assume that a component $B$ of $G-V(C)$ contains a vertex $u$ of $R$;  the case where a component $B$ of $G-V(C)$ contains a vertex $v$ in $S$ is similar. If a neighbor of $B$ belongs to $C(v_1,v_2)$ and a neighbor of $B$ belongs to $C(v_2,v_1)$, then, by Lemma~\ref{lem:K23d0},  $c(u_1u_2;v_1v_2)\in P(G; R, S)$. 

We may therefore assume that $N(B)\subseteq C[v_2,v_1]$. Suppose $N(B)=\{v_1,v_2\}$. By Lemma~\ref{lem:midsep}, $c(u_1u_2;v_1v_2)\in P(G; R,S)$ or $G$ has a sided $(R,S)$-separation or a sided $(S,R)$-separation.

We may therefore assume that a neighbor of $B$ belongs to $C(v_2,v_1)$. 
By Lemma~\ref{lem:misses1ofv1v2}, $c(u_1u_2;v_1v_2)\in P(G; R, S)$,  $G$ has a sided $(R,S)$-separation or a sided $(S,R)$-separation, or there exists a cyclic ordering $w_1,\ldots,w_{n+1}$ with $\{t_1,\ldots,t_n\}\subset \{w_1,\ldots,w_{n+1}\}\subseteq R\cup S$ such that $(G, w_1,\ldots,w_{n+1})$ is nearly planar and $t_1,\ldots,t_n$ occur in this ordering in $w_1,\ldots,w_{n+1}$. 
\end{proof}

\begin{lemma}\label{lem:Kuratowskiconnected2}
Let $G$ be a $2$-connected graph and let $R, S\subseteq V$, where $R\cup S$ has at least four vertices. Let $u_1,u_2,u_3\in R\cap S$ be distinct vertices. Suppose $c(u_1, u_2, u_3)\not\in P(G; R, S)$. Then there exist distinct vertices $u_4,u_5\in R$ and distinct vertices $v_1,v_2\in S$ such that $c(u_4u_5;v_1v_2)\not\in P(G; R, S)$ and $c(u_1, u_2, u_3)+c(u_4u_5;v_1v_2)\in P(G; R, S)$. 
\end{lemma}
\begin{proof}
Since $R\cup S$ has at least four vertices, there exists a vertex $v\in R\cup S$ distinct from $u_1,u_2,u_3$. By symmetry, we may assume that $v\in R$. Since $G$ is $2$-connected, there exist two vertex-disjoint paths $P_1,Q_1$ from $\{v,u_1\}$ to $\{u_2,u_3\}$ and two vertex-disjoint paths $P_2,Q_2$ from $\{v,u_2\}$ to $\{u_1,u_3\}$. Let $v$ and $z_2\in \{u_2,u_3\}$ be the ends of $P_1$, and let $z_1$ and $z_3$ be the ends of $Q_1$. Let $v$ and $z_1$ be the ends of $P_2$, and let $z_2$ and $z_3$ be the ends of $Q_2$. Then $\{z_1,z_2,z_3\} = \{u_1,u_2,u_3\}$, and
\begin{equation*}
c(vz_2; z_1z_3) + c(vz_1; z_3z_2) = c(z_1,z_2,z_3) + c(vz_3; z_1z_2).
\end{equation*}
Since either $c(z_1,z_2,z_3) = c(u_1, u_2, u_3)$ or $c(z_1,z_2,z_3) = -c(u_1, u_2, u_3)$, $c(z_1,z_2,z_3)\not\in P(G; R, S)$. Since $c(vz_2; z_1z_3), c(vz_1; z_3z_2)\in P(G; R, S)$, we obtain that $c(vz_3; z_1z_2)\not\in P(G; R, S)$.
\end{proof}

In the proof of the Theorem~\ref{thm:3planarboundary}, we use the following theorem; see \cite{Seymour80}.

\begin{theorem}\label{thm:3planar}
Let $G=(V,E)$ be a graph and let $s_1,s_2,t_1,t_2$ be vertices of $G$. Then $G$ contains no  $(s_1t_1;s_2t_2)$-linkage if and only if $(G; s_1, s_2, t_1, t_2)$ is nearly planar.
\end{theorem}

\begin{proof}[Proof of Theorem~\ref{thm:3planarboundary}]
Since $L(R,S)-P(G; R, S)\not=\emptyset$, there exist vertices $u_1,u_2\in R$ and vertices $v_1,v_2\in S$ such that $c(u_1u_2;v_1v_2)\in L(R,S)-P(G; R, S)$, or there exists vertices $u_1,u_2,u_3\in R\cap S$ such that $c(u_1, u_2, u_3)\in L(R, S)-P(G; R, S)$, by Lemma~\ref{lem:circuitandkurcycles}. 

If $c(u_1, u_2, u_3)\in L(R, S)-P(G; R, S)$ and $|R\cup S|=3$, then $(G; u_1,u_2,u_3)$ is nearly planar and $R$ and $S$ cross.
If $c(u_1,u_2, u_3)$ and $|R\cup S|>3$, then there exist vertices $u_4, u_5\in R$ and vertices $v_1,v_2\in S$ such that $c(u_4u_5;v_1v_2)\in L(R, S)-P(G; R, S)$, by Lemma~\ref{lem:Kuratowskiconnected2}. Hence we may assume that $c(u_1u_2; v_1v_2)\not\in P(G; R, S)$. 
By Theorem~\ref{thm:3planar}, $(G; u_1, v_1, u_2, v_2)$ is nearly planar. By Lemma~\ref{internRandS}, either $G$ has a sided $(R,S)$-separation or a side $(S,R)$-separation, or there exists a cyclic ordering $w_1,\ldots,w_n$ of the vertices in $R\cup S$ such that $u_1,v_1,u_2,v_2$ occur in this order in $w_1,\ldots,w_n$ and $(G, w_1,\ldots,w_n)$ is nearly planar.
Also $R$ and $S$ cross.
\end{proof}

\section{Proof of Theorem~\ref{thm:K33}}

\begin{lemma}
Let $(G, \mathcal{A}, r_1,\ldots,r_n)$ be a $2$-connected nearly planar graph. If $u_1,u_2,v_1,v_2$ are distinct vertices occurring in this order in $r_1,\ldots,r_n$, then there exists a $(u_1u_2; v_1v_2)$-linkage in $G$.
\end{lemma}

\begin{lemma}\label{lem:skewK33}%%checked
Let $(G, \mathcal{A}, t_1,\ldots,t_n)$ be a $2$-connected nearly planar graph, and let $R,S \subseteq \{t_1,\ldots,t_n\}$.
Suppose $u_1,v_1,u_2,v_2\in R\cap S$. If $u_1,v_1,u_2,v_2$ are distinct vertices occurring in this order in $t_1,\ldots,t_n$, then $c(u_1u_2;v_1v_2) + c(v_1v_2;u_1u_2) \in P(C; R, S)$.
\end{lemma}
\begin{proof}
Since $(G, \mathcal{A}, t_1,\ldots,t_n)$ is $2$-connected, $G$ contains a $(u_1v_1; u_2v_2)$-linkage, a $(v_1u_2; v_2u_1)$-linkage, a $(u_2v_2; u_1v_1)$-linkage,  and a $(v_2u_1; v_1u_2)$-linkage Hence $c(v_1u_1;v_2u_2)\in P(G; R, S)$, $c(u_2v_1;u_1v_2)\in P(G; R, S)$, $c(v_2u_2;v_1u_1)\in P(G; R, S)$, and $c(u_1v_2;u_2v_1)\in P(G; R, S)$. Then 
\begin{equation*}
c(u_1u_2;v_1v_2) + c(v_1v_2;u_1u_2) = c(v_1u_2;v_2u_2) - c(u_2v_1;u_1v_2) + c(v_2u_2; v_1u_1) - c(u_1v_2;u_2v_1),
\end{equation*}
and hence $c(u_1u_2;v_1v_2) + c(v_1v_2;u_1u_2)\in P(C; R, S)$.
\end{proof}

\begin{lemma}\label{lem:Kuratowskiconnected}
Let $(G, \mathcal{A}, t_1,\ldots,t_n)$ be a $2$-connected nearly planar graph, and let $R,S \subseteq \{t_1,\ldots,t_n\}$. Suppose that $u_1,u_2,\overline{u}_1,\overline{u}_2\in R$ and $v_1,v_2,\overline{v}_1,\overline{v}_2\in S$.
If $u_1,v_1,u_2,v_2$ are distinct vertices, occurring in this order in $t_1,\ldots,t_n$, and $\overline{u}_1,\overline{v}_1,\overline{u}_2,\overline{v}_2$ are distinct vertices, occurring in this order in $t_1,\ldots,t_n$, then 
\begin{equation*}
c(u_1u_2;v_1v_2)-c(\overline{u}_1\overline{u}_2;\overline{v}_1\overline{v}_2) \in P(G; R, S).
\end{equation*}
\end{lemma}

\begin{proof}
Suppose for a contradiction that the statement of the lemma is false.
Choose the sequences  $u_1,v_1,u_2,v_2$ and $\overline{u}_1,\overline{v}_1,\overline{u}_2,\overline{v}_2$ such that $c(u_1u_2;v_1v_2)-c(\overline{u}_1\overline{u}_2;\overline{v}_1\overline{v}_2) \not\in P(G; R, S)$.

If $\{u_1,u_2\} = \{\overline{v}_1,\overline{v}_2\}$ and $\{v_1,v_2\}= \{\overline{u}_1,\overline{u}_2\}$, then, by Lemma~\ref{lem:skewK33}, $c(u_1u_2;v_1v_2)-c(\overline{u}_1\overline{u}_2;\overline{v}_1\overline{v}_2)\in P(G; R, S)$. We may therefore assume that $\{u_1,u_2\}\not=\{\overline{v}_1,\overline{v}_2\}$ or $\{v_1,v_2\}\not=\{\overline{u}_1,\overline{u}_2\}$. By symmetry, we may assume that $u_1 \not\in \{\overline{v}_1,\overline{v}_2\}$. Furthermore, by symmetry we may also assume that $u_1$ is a vertex in between $\overline{v}_2$ and $\overline{v}_1$ in the cyclic ordering $t_1,\ldots,t_n$. Since $(G, \mathcal{A}, t_1,\ldots,t_n)$ is $2$-connected, $G$ contains a $(u_1\overline{u}_1; \overline{v}_1\overline{v}_2)$-linkage. Hence 
\begin{equation*}
c(u_1\overline{u}_1;\overline{v}_1\overline{v}_2)\in P(G; R, S),
\end{equation*}
 and therefore 
 \begin{equation*}
 c(u_1u_2;v_1v_2) - c(\overline{u}_1\overline{u}_2;\overline{v}_1\overline{v}_2)-c(u_1\overline{u}_1;\overline{v}_1\overline{v}_2)\not\in P(G; R, S).
 \end{equation*}
Since
 \begin{equation*}
 c(u_1u_2;v_1v_2) - c(u_1\overline{u}_2;\overline{v}_1\overline{v}_2) = 
 c(u_1u_2;v_1v_2)  - c(\overline{u}_1\overline{u}_2;\overline{v}_1\overline{v}_2) -c(u_1\overline{u}_1;\overline{v}_1\overline{v}_2),
 \end{equation*}
 we obtain
 \begin{equation*}
  c(u_1u_2;v_1v_2) - c(u_1\overline{u}_2;\overline{v}_1\overline{v}_2) \not\in P(G; R, S).
 \end{equation*}
 We may therefore assume that $u_1=\overline{u}_1$. Since $u_1,v_1,u_2,v_2$ and $u_1,\overline{v}_1,\overline{u}_2,\overline{v}_2$ occur in the same ordering in $t_1,\ldots,t_n$, either $v_1$ is between $u_1$ and $\overline{u}_2$ in the cyclic ordering $t_1,\ldots,t_n$ or $v_2$ is between $\overline{u}_2$ and $u_1$ in the cyclic ordering $t_1,\ldots,t_n$. We may assume that $v_1$ is between $u_1$ and $\overline{u}_2$ in the cyclic ordering $t_1,\ldots,t_n$. Since $(G, \mathcal{A}, t_1,\ldots,t_n)$ is $2$-connected, $G$ contains a $(u_1\overline{u}_2; v_1\overline{v}_1)$-linkage. Hence
 \begin{equation*}
 c(u_1\overline{u}_2;v_1\overline{v}_1)\in P(G; R, S),
 \end{equation*}
and therefore
\begin{equation*}
 c(u_1u_2;v_1v_2) - c(u_1\overline{u}_2;\overline{v}_1\overline{v}_2) - c(u_1\overline{u}_2;v_1\overline{v}_1)\not\in P(G; R, S).
\end{equation*}
Since 
\begin{equation*}
c(u_1u_2;v_1v_2) - c(u_1\overline{u}_2;v_1\overline{v}_2) = c(u_1u_2;v_1v_2) - c(u_1\overline{u}_2;\overline{v}_1\overline{v}_2) - c(u_1\overline{u}_2;v_1\overline{v}_1),
\end{equation*}
we obtain
\begin{equation*}
c(u_1u_2;v_1v_2) - c(u_1\overline{u}_2;v_1\overline{v}_2) \not\in P(G; R, S).
\end{equation*}
We may therefore assume that $\overline{v}_1 = v_1$. In the same manner, $u_2=\overline{u}_2$ and $v_2=\overline{v}_2$. This is a contradiction.
\end{proof}

\begin{proof}[Proof of Theorem~\ref{thm:K33}]
$(\ref{item:i1})\implies (\ref{item:i2})$ Suppose $L(R, S)-P(G; R, S)\not=\emptyset$. Since $G$ is $(R,S)$-connected, there exists a cyclic ordering $t_1,\ldots,t_n$ of the vertices of $R\cup S$ such that  $(G, t_1,\ldots,t_n)$ is nearly planar and $R$ and $S$ cross, by Theorem~\ref{thm:3planarboundary}.

$(\ref{item:i2})\implies (\ref{item:i3})$ Suppose $(G, t_1,\ldots,t_n)$ is nearly planar and $R$ and $S$ cross. Suppose first that $R=S$ and $|R|=3$. Then $P(G; R, S)=\{0\}$ and $L(R, S)$ is generated by $c(u_1,u_2,u_3)$, where $u_1,u_2,u_3$ are the vertices in $R$. Since $\link(c(u_1,u_2,u_3); t_1,\ldots,t_n)\in \{-2,2\}$, we see that in this case, for all $z\in L(R, S)$, $\link(z; t_1,\ldots, t_n) = 0$ if and only if $z \in P(G; R, S)$. Suppose next that $|R\cup S|>3$. Let $c\in L(R, S)$ with $\link(c; t_1,\ldots,t_n) = 0$. By Lemma~\ref{lem:circuitandkurcycles}, $c = \sum_{i=1}^m c_i$, where each $c_i$  is an elementary $(R, S)$-tensor. Observe that $\link(c; t_1,\ldots,t_n) = \sum_{i=1}^m \link(c_i; t_1,\ldots, t_n)$ and, for each elementary $(R, S)$-tensor $d$, $\link(d; t_1,\ldots,t_n)\in \{-2,0,2\}$. If $\link(c_j; t_1,\ldots,t_n)\not=0$ for some $j\in \{1,\ldots,m\}$, then there exists an $c_k$ such that $\link(c_k; t_1,\ldots,t_n) = -\link(c_j; t_1,\ldots,t_n)$. Then $\link(c_j-c_k; t_1,\ldots,t_n) = 0$. By Lemmas~\ref{lem:Kuratowskiconnected} and \ref{lem:Kuratowskiconnected2}, $c_j-c_k\in P(G; R, S)$. Hence $c\in P(G; R, S)$.

Conversely, if $c\in P(G; R, S)$, then $c = \sum_{i} c_i$, where each $c_i$ is of the form $c(u_1u_2;v_1v_2)$ with $u_1,u_2,v_1,v_2$ distinct, and $u_1,u_2\in R$ and $v_1,v_2\in S$, such that there exists an $(u_1u_2; v_1v_2)$-linkage in $G$. Since $\link(c(u_1u_2;v_1v_2); t_1,\ldots,t_n) = 0$ if there exists an $(u_1u_2; v_1v_2)$-linkage in $G$, $\link(c; t_1,\ldots,t_n) = 0$.

$(\ref{item:i3})\implies (\ref{item:i1})$ Suppose that there exists a cyclic ordering $t_1,\ldots,t_n$ of the vertices in $R\cup S$ such that $R$ and $S$ cross and, for all $z\in L(R, S)$, $\link(z; t_1,\ldots, t_n) = 0$ if and only if $z \in P(G; R, S)$. Since $R$ and $S$ cross, either $R=S$ and $|R|=3$, or there exist vertices $u_1,u_2\in R$ and $v_1,v_2\in S$ such that $u_1,v_1,u_2,v_2$ are distinct and occur in this order in the cyclic ordering $t_1,\ldots,t_n$. If $R=S$ and $|R|=3$, let $u_1,u_2,u_3$ be the vertices of $R$. Then $c(u_1,u_2,u_3)\not\in P(G; R, S)$. Suppose next that there exist vertices $u_1,u_2\in R$ and $v_1,v_2\in S$ such that $u_1,v_1,u_2,v_2$ are distinct, occurring in this order in the cyclic ordering $t_1,\ldots,t_n$. Since $\link(c(u_1u_2;v_1v_2); t_1,\ldots,t_n)\not=0$, $c(u_1u_2;v_1v_2)\not\in P(G; R, S)$.
\end{proof}

\bibliographystyle{plain}
\bibliography{biblio}
\end{document}